\documentclass[reqno]{amsart}
\usepackage{enumerate}
\usepackage[mathscr]{eucal}
\newtheorem{theorem}{Theorem}[section]
\newtheorem{lemma}[theorem]{Lemma}
\newtheorem{proposition}[theorem]{Proposition}

\theoremstyle{definition}
\newtheorem{definition}[theorem]{Definition}

\newtheorem{remark}[theorem]{Remark}

\newcommand{\norm}[1]{\left\Vert#1\right\Vert}

\newcommand{\tap}[1]{\left\{#1\right\}}
\usepackage{lineno}

\numberwithin{equation}{section}
\usepackage{amsfonts}
\usepackage{color}

\usepackage{amsmath}
\usepackage{amssymb}

\begin{document}
\font\nho=cmr10
\def\dive{\operatorname{div}}
\def\cal{\mathcal}
\def\L{\cal L}
\def\p{\mathbb P}
\def\RR{\mathcal{R}}

 \def\le{\leqslant} 
 \def\leq{\leqslant} 
\def\rar{\rightarrow}
\def\Rar{\Rightarrow}
\def\f{\mathbf F}
\def\r{\hro{R}}
\def\e{\cal{E}}
\def\B{\cal B}
\def\A{\mathcal{A}}
\def\p{\mathbb P}
 
\def\grad{\operatorname{grad}}
\def\dive{\operatorname{div}}

\def\tet{\theta}
\def\Tet{\Theta}
\def\hro{\mathbb}
\def\ho{\mathcal}
\def\P{\ho P}
\def\T{\Delta_H}
\def\E{\mathcal{E}}
\def\n{\mathbb{N}}
\def\M{\mathbb{S}^d}
\def\dMu{\mathbf{U}}
\def\dMcs{\mathbf{C}}
\def\dMcu{\mathbf{C^u}}
\def\vk{\vskip 0.2cm}
\def\td{\Leftrightarrow}
\def\df{\frac}
\def\Wei{\mathrm{We}}
\def\Rey{\mathrm{Re}}
\def\s{\mathbb S}
\def\l{\mathcal{L}}
\def\C+{C_+([t_0,\infty))}
\def\o{\cal O}
\def\D{\mathbf{\Delta}}
\def\grad{\operatorname{grad}}

\title[Navier-Stokes on $d$-Spheres]{Periodicity and Inertial Manifolds for Navier-Stokes Flows on $d$-Spheres}
\author[T.N.H. Vu]{Thi Ngoc Ha Vu }
\address{Thi Ngoc Ha Vu \hfill\break
Faculty of Mathematics and Informatics, 
 Hanoi University of Science and Technology,
 Khoa Toan - Tin, Dai hoc Bach khoa Hanoi,  
 1 Dai Co Viet, Bach Mai, Hanoi, Vietnam}
 \email{ha.vuthingoc@hust.edu.vn}
 \author[T.H. Nguyen]{Thieu Huy Nguyen $^*$}
\thanks{$^*$ Corresponding author: Thieu Huy Nguyen. Email: huy.nguyenthieu@hust.edu.vn}
\address{Thieu Huy Nguyen \hfill\break
Faculty of Mathematics and Informatics, 
 Hanoi University of Science and Technology,
 Khoa Toan - Tin, Dai hoc Bach khoa Hanoi,  
 1 Dai Co Viet, Bach Mai, Hanoi, Vietnam}
 \email{huy.nguyenthieu@hust.edu.vn}

\begin{abstract}   Consider 
a   $d$-sphere $\M$ and denote by $\Gamma(T\M)$ the set of all vector fields on $\M$. We study    
 the Navier-Stokes  equations (NSE): 
$$
\partial_t u + \nabla_u u + \grad\pi  = \D u + \dive f(\cdot, t);\, \dive u=0,$$
for the vector field $u(\cdot, t)\in  \Gamma(T\M)$, where $\D$ denotes the Ebin-Marsden Laplace operator defined by $\D u=  \dive (\nabla u + \nabla u^t)^{\sharp}$, and $\dive f(\cdot, t)$ is the periodic external force.    
 We investigate the Navier-Stokes equations in the framework of $L^p$-spaces over vector fields on $\M$ and prove the existence and  uniqueness of a periodic solution to such equations. Moreover, exploiting the distribution of eingenvalues of Ebin-Marsden Laplace operator we show  the existence of an inertial manifold for solutions around that solution.  
\end{abstract}



\maketitle
\section{Introduction and Preliminaries}

The notion of inertial manifolds associated with solutions to parabolic equations of the form
$u'=Au+f(u)$
has initially been posed by C. Foias, G.R. Sell, and R. Temam \cite{FST} within the framework of efforts aimed at advancing the mathematical understanding of the Navier-Stokes equations via the reduction principle. This approach sought to reduce the  infinite-dimensional dynamical systems to finite-dimensional approximations, thereby facilitating a more rigorous analysis of their qualitative and quantitative behavior. In their seminal work \cite{FST}, the linear operator 
$A$ was typically taken to be the Laplacian defined over an appropriate domain and function space, ensuring that 
$A$ was a self-adjoint operator with compact resolvent and a spectrum consisting of a sequence of eigenvalues tending to infinity. Furthermore, the construction of inertial manifolds crucially relied upon the satisfaction of certain spectral gap conditions, which guarantee the necessary separation of modes for the finite-dimensional reduction to be valid.

Since its inception, the concept of inertial manifolds has become a fundamental tool in the analysis of the asymptotic behavior of solutions to a wide class of evolution equations. In particular, it has enabled the rigorous derivation of finite-dimensional models that accurately capture the long-time dynamics of infinite-dimensional systems. Substantial progress has been achieved in proving the existence, regularity, and structural properties of inertial manifolds for various classes of parabolic equations and related systems, as documented in a broad body of literature including \cite{MCR,   ChueshovScheutzow, CFNT, KS, Temam, Ze14}. These developments have significantly deepened the theoretical understanding of dissipative dynamical systems and have provided an analytical framework for  problems arising in fluid dynamics, reaction-diffusion systems, and beyond.

In recent developments, the theory of inertial manifolds has been successfully extended to address a variety of problems beyond its original scope, including applications to feedback control systems \cite{You-1993} and the analysis of reaction-diffusion equations (\cite{RosaTem}, \cite{RosaJDDE2003}, \cite{KS2011}). Moreover, significant efforts have been directed toward generalizing the concept to accommodate models arising in meteorology \cite{DebTem91} as well as evolution equations governed by sectorial operators (see, e.g., \cite{Sell-jde1992}, \cite{Takagi08}). These extensions underscore the conceptual flexibility of inertial manifold theory and its capacity to provide finite-dimensional reductions for a broad class of infinite-dimensional dynamical systems. In the context of semilinear evolution equations with a nonlinear term that is 
$\varphi$-Lipschitz, where 
$\varphi$ belongs to an appropriate admissible class of functions, Nguyen \cite{Huy2012} has established the existence of inertial manifolds under the assumption that the associated linear operator is self-adjoint in a Hilbert space and possesses a compact resolvent. Furthermore, a refinement of the classical concept has been proposed in \cite{Huy2013}, wherein the existence of \textit{admissibly inertial manifolds of  
$\mathcal E$-class} has been rigorously demonstrated. These results represent substantial progress toward extending the applicability of inertial manifold techniques to broader functional analytic settings and more general classes of evolution equations. Recently, for the case of diffusion equations on $d$-spheres with the linear part being Hodge-Laplacian, we have proved the existence of an inertial manifolds around a periodic solution of such equations (see \cite{NV2025}).

Nevertheless, despite these advances, the question of the existence of inertial manifolds for the Navier-Stokes equations remains unresolved and constitutes a major open problem in the theory of dissipative systems (cf. Kwak \cite{Kwak-NSE}). 
 In the present paper, in an attempt to search for a circumstance for the possibility of the existence of an inertial manifold for Navier-Stokes equations,   we first investigate  the existence of a periodic solution to the Navier-Stokes equation for vector fields on a 
$d$-sphere $\M$ under the action of an periodic external force. Then,
  we study the existence of an inertial manifold associated with such a  solution. Concretely, denoting $T\M$ the tangent bundle on $\M$, and $\Gamma(T\M)$ the set of all vector fields on $\M$,
  we consider the Navier-Stokes equations for  $u:\M\times \r\to T\M$
  \begin{align}\label{dff1}
\partial_t u + \nabla_u u + \grad\pi  = \D u + \dive f(t);\, \dive u=0,
\end{align}
where $\D$ is the Ebin-Marsden Laplace operator (defined below); $f$  is the external force. 
 To formulate Navier-Stokes equation in  $\M$, it is necessary to generalize the standard differential operators,  the divergence of a tensor field, the Laplacian with respect to the metric. This can be done using the Levi-Civita connection. We will briefly recall  the differential operators used in the formulation of the diffusion equations on $\M$ such as $\grad, \dive$,... etc, in the coming paragraphs.

\def\g{\mathbf{g}}
We  recall some preparatory notions and result for latter use. Consider a $d$-dimensional compact Riemannian manifold $(M,\mathbf{g})$. We refer to \cite{He,Jo} 
for details on Riemannian manifolds and  related concepts of geometric analysis. 

Denote by $\nabla$  the Levi-Civita connection and the set of vector fields on $M$ by $\Gamma(TM)$.  
We study the Navier-Stokes equations on $M$ given by
$$\partial_t u + \nabla_u u + \grad\pi  =  \mathbf{\Delta} u + \dive f;\, \dive u=0,\hbox{ for Ebin-Marsden Laplacian $\D$ defined below,} $$
here the unknown velocity $u:M\times \r \to TM$ 
is considered as a vector field on $M$, i.e., 
$u(\cdot, t)\in \Gamma(TM)$, $\dive f$  is the external force.  
  Related notions of differential operators on manifolds are now recalled. Firstly, for $X\in \Gamma(TM)$ we then extend  $\nabla_X$ to any $(p,q)$ 
  tensor as follows: 
  for 
$S\in \Gamma(\otimes^p(TM)\otimes^q
(T^*M))$  we set 
$$(\nabla_X S)(X_1, \cdots, X_q ) = \nabla_X(S(X_1, \cdots, X_q ))-
S(\nabla_XX_1,\cdots, X_q))-\cdots-S(X_1, \cdots,\nabla_X X_q).$$
Moreover, the covariant derivative $\nabla$ on tensor field 
$S\in \Gamma(\otimes^p(TM)\otimes^q(T^*M))$  is defined  
by
$$\nabla S(X, X_1, \cdots, X_q) = (\nabla _X S)(X_1, \cdots, X_q),
\hbox{ hence, }\nabla S\in \Gamma(\otimes^p(TM)\otimes^{q+1}(T^*M).$$
Next, adapting  the ``music notations", for an 1-form $w$,  the vector field $w^\sharp$ is defined by 
$$\g(w^\sharp, Y ) = w(Y), \forall Y \in \Gamma(TM)$$
whereas, for a vector field $X$,  the 1-form $X^\flat$ can be defined as 
$$X^\flat(Y) = \g(X, Y ), \forall Y \in \Gamma(TM).$$
The metric on 1-forms  is defined by setting
$\g(w, \eta) := \g(w^\sharp, \eta^\sharp), \forall w, \eta \in \Gamma(T^*M).$\\
The gradient of a function is then defined as
  $\grad p = (dp)^\sharp.$
  
Generally, for $(p,q)$-tensor field $ T\in \Gamma(\otimes^p(TM)\otimes^q(T^*M))$
  we have
\begin{eqnarray*}
T^\sharp& = &C^2_1(\g^{-1}\otimes T) \in \Gamma(\otimes^{p+1}(TM)\otimes^{q-1}(T^*M)),\\
T^\flat & = &C^1_2(\g \otimes  T ) \in \Gamma(\otimes^{p-1}(TM)\otimes^{q+1}(T^*M)),\\
\dive T &= &C^1_1\nabla T \in \Gamma(\otimes^{p-1}(TM)\otimes^{q}(T^*M))
\end{eqnarray*}
where $C^i_j$ stands for the contraction of the $i$ and $j$ indices for tensors.

   Next,  the Hodge Laplacian on
   1-forms denoted by $\Delta_H$ is defined by 
$$\Delta_H = d^*d + dd^*. $$
Then,  the Ebin-Marsden Laplacian $\D$ is defined by the stress tensor (see \cite{EbiMa,Tay}) as:
$$\D u = \dive(\nabla u + \nabla u^t)^{\sharp}.$$
Since $\dive u=0$ we can express $\D
$ in the following formula
$$\D u = -\T u + 2R(u),$$
where 
  $R(\cdot)$ is the Ricci operator related to the Ricci curvature tensor ($\mathrm{Ric}$) by the formula
$$R(u) = (\mathrm{Ric}(u,\cdot))^{\sharp} \hbox{ for all }u\in\Gamma({TM}).$$
 
We remark that on a $d$-dimensional sphere $\M:=\{x\in \r^{d+1} : |x|=1\}$, the Ricci  curvature tensor has the form
$\mathrm{Ric}=(d-1)\g$.
 Therefore,   $R(u)=(d-1)u$, and hence $\D u = (-{\Delta}_Hu^\flat)^\sharp+2(d-1)u $. Now,  
we consider the Navier-Stokes equations  on $\M$: 
\begin{align}\label{Cauchy}
\begin{cases}
\partial_t u + \nabla_u u + \grad\pi  =  \mathbf{\Delta} u + \dive f;\\
\dive u=0.\end{cases}
\end{align}
We note that, since $\dive u=0$ it follows that $\nabla_uu = \dive (u\otimes u).$ Using Hodge decomposition and the Helmholtz-Hodge projection $\p$ onto the space $L^p_\sigma(\Gamma(T\M)):=
\overline{\Lambda^1_{c, \sigma}(\M)}^{\|\cdot\|_{L^p}}$ 
where $\Lambda^1_{c, \sigma}(\M)$ denotes the space of all smooth, divergence free, compactly supported vector fields on $\M$ (note that  $\ker \p=\tap{\nabla f\in L^p(\Gamma(T\M)) : f\in L^p_{loc}(\Gamma(T\M))}$) we can get rid of the pressure term $\pi$ to obtain
\begin{align}\label{NSESd1}
\begin{cases}
\partial_t u + \p\dive (u\otimes u)  =  \mathbf{\Delta} u + \p\dive f;\\
u\in L^p_\sigma(\Gamma(T\M)).\end{cases}
\end{align}
Note that, on $L^p_\sigma(\Gamma(T\M))$ we have $\p\D=\D$. Moreover, 
 on $\M$, the spectrum of ${\Delta}_H$ consists of eigenvalues $ \lambda_k= k(k+d-1)+d-1$ for $k=1, 2, \ldots$   Therefore, the spectrum of $-\D={\Delta}_H-2(d-1)I$ appears as  
 $$\tap{\lambda_k=k^2+(k-1)(d-1) : k=1, 2, \ldots}.$$

This distribution of eigenvalues of $-\D$ leads to the validity of spectral gap condition for certain choices of $k, d$, yielding 
the existence of an inertial manifold.
Denote by $(e^{t{\D}})_{t\geqslant 0}$ the semigroup generated by ${\D}$, i.e., the semigroup associated with the homogeneous Cauchy problem
\begin{align}\label{CauchyStokes0}
\begin{cases}
\partial_t u &= {\D}u,\\
u|_{t=0} &= u_0 \in \Gamma(T\M), 
\end{cases}
\end{align}
\textit{i.e.}, the unique solution of the above Cauchy problem is given by $u(t)=e^{t\D}u_0$.

 Assume that $\{e_{k}\}_{k=1}^{\infty}$ is the orthonormal basis of $X:=L^2(\Gamma(T\M))$ consisting of the corresponding eigenfunctions of the operator $-\D$ (i.e., $-\D e_k=\lambda_k e_k$).
Let now $\lambda_{N}$ and $\lambda_{N+1}$ be two successive and different eigenvalues with $\lambda_{N} < \lambda_{N+1}$, let further $P$ be the orthogonal projection onto the first $N$ eigenvectors of the operator $-\D$. 
Since $PX$ is of finite dimension, we have that the restriction $(e^{t\D}P)_{t\geqslant 0}$ of the semigroup $(e^{t\D})_{t\geqslant 0}$ to $PX$ can be extended to the whole line $\mathbb{R}$.

Furthermore, 
the following exponential decay estimate for $(e^{t\D})_{t\geqslant 0}$ has been essentially proved in \cite{Pi}.

\begin{proposition}\label{lplq}
Let $ 1\le p< \infty$ and denote by $\|\cdot\|_{p}$ the norm in the space $L^{p}(\Gamma(T\M))$. 
\begin{enumerate}[(i)]
\item For $t>0$  
we have 
\begin{equation}
\label{dispersive}
\left\| e^{t\Delta} x\right\|_{L^p} \leq Me^{-\alpha t}\left\| x \right\|_{L^p}\hbox{ for all }x \in L^p(\Gamma(T\M)),
\end{equation}
 where $M$ and $\alpha$ are certain positive constants.  
\item  For $p$ and $q$ such that $1<p\leq q <\infty$, and for all $t>0$ we have
\begin{equation} \label{lplq1}
\left\| e^{t\D}\dive f \right\|_{L^q} \leq [h_d(t)]^{\frac{1}{p}-\frac{1}{q}+\frac{1}{d}}e^{-t\beta} \left\| f\right\|_{L^p}
\hbox{ for all  $f \in L^p(\Gamma(T\M\otimes T\M))$}
\end{equation} 
 where $h_d(t): = C\max\left( \frac{1}{t^{d/2}},1 \right)$, 
   and $\beta$ is certain positive constant depending only on $p, q$, and  $d$.
\end{enumerate}
\end{proposition}
\begin{proof}
See \cite{Pi}.  
\end{proof}

\section{Main Results}
We formulate our main results on  periodic solutions and the associated inertial manifold in their vicinity. For 
 a Banach space $X$ with norm $\norm{\cdot}_X$ we denote by 
$${C}_b(\r, X):=\tap{v:\r\to X  \mid  v \hbox{ is continuous and }\sup_{t\in \r}\|v(t)\|_{X}<\infty}$$
endowed with sup-norm $\|v\|_{C_b}:=\sup_{t\in \r}\|v(t)\|_{X}$. 

We investigate the semilinear  equation
\begin{align}\label{NSESd}
\begin{cases}
\partial_t u(t) + \p\dive (u\otimes u)  =  \mathbf{\Delta} u + \p\dive f,\ \ t>s;\\
u(s)=x\in \Gamma(T\M)\end{cases}
\end{align}
where 
the external force $f\in {C}_b(\r, X)$ is $T$-periodic with respect to time $t$.  
    Furthermore, by the {\it mild solution} to 
  (\ref{NSESd})  we mean the function $u$ satisfying the following integral equation (Duhamel's formula): 
\begin{equation}\label{mild}
u(t)= e^{(t-s)\D}x+\int_s^te^{(t-\tau)\D}\p\dive[-u(\tau)\otimes u(\tau)+f(\tau)]d\tau\hbox{\ for all }t\ge s.
\end{equation}
We are stating the first main result on existence of a unique $T$-periodic mild solution to \eqref{NSESd} under the action of a $T$-periodic external force.

\begin{theorem}\label{mainnon}
  Consider Navier-Stokes equations \eqref{NSESd}  and  $f\in C_b(\r, L^{p/2}(\Gamma(T\M\otimes T\M)))$. Suppose that $f$ is  $T$-periodic with respect to time $t$. 
 Then, if $\|f\|_{C_b(\r, L^{p/2}(\Gamma(T\M\otimes T\M)))}$ is sufficiently  small, Equation (\ref{NSESd}) has one and only 
one   $T$-periodic  mild solution $\hat {u}$ on a certain ball of $C_b (\mathbb{R}, L_\sigma^p(\Gamma(T\M)))$.
\end{theorem}

Next, we will state our main result on the existence of the inertial manifold 
 for the mild solutions to 
 Equation 
 \eqref{NSESd} around a periodic mild solution.  
To that purpose, we first introduce the concept of inertial manifolds near a certain mild periodic solution in the following definition.
\begin{definition}\label{Inertial} Consider the projection $P$ as above, and 
  let $\hat{u}$ be a periodic mild solution to \eqref{NSESd}.
An \emph{inertial manifold}  for mild solutions to \eqref{NSESd} around  $\hat{u}$ is a Lipschitz surface $\mathbb{M}$ in $X:=L_\sigma^2(\Gamma(T\M))$ which  is the graph of a Lipschitz function $\Phi: \mathrm{Im} P \rightarrow (I-P)X$, i.e.,
$$\mathbb{M} = \left\{ x+\Phi x \mid x\in \mathrm{Im}P \right\},$$
and the following conditions are satisfied:
\begin{itemize}
\item[(i)] $\mathbb{M}$ is positively $\hat{u}$-invariant under Eq. \eqref{NSESd}, i.e., if a mild  solution $x(t), t\geqslant s$, of Eq. \eqref{NSESd} satisfies $x(s)-\hat{u}(s) \in \mathbb{M}$, 
then we have that $x(t)-\hat{u}(t) \in  \mathbb{M}$ for $t\geqslant s$.
\item[(ii)]  $\mathbb{M}$  exponentially $\hat{u}$-attracts all the mild solutions to Eq. \eqref{NSESd} in the sense that for each mild solution $u(\cdot)$ to  Eq. \eqref{NSESd} and any fixed $s\in \mathbb{R}$, 
 there is a mild solution $u^* (\cdot)$ of \eqref{NSESd} such that $u^*(t)\in \mathbb{M}$ for $t\geqslant s$ and
\begin{equation}\label{attract}
\norm{u(t)-\hat{u}(t)-u^*(t)} \leqslant Ce^{-\gamma(t-s)} \hbox{  for all  } t\geqslant s \hbox{  and some constants  } \gamma, C > 0.
\end{equation}
\end{itemize}
\end{definition}
For simplicity of presentation  we put
 $
A:=-\D 
$ on the space $
H:=L^2_\sigma(\Gamma(TS^d)).
$
Then \(A\) is a positive self-adjoint operator with compact inverse. Let 
$$
0<\lambda_1\leq\lambda_2\leq\cdots
$$ 
be the eigenvalues of \(A\), counted with multiplicity.

Let 
$
\widehat u\in C_b(\mathbb R,H^1(\Gamma(TS^d)))
$ be a periodic solution and take any other solution $u$. Then,  
for \(w=u-\widehat u\), the Navier-Stokes equation can be written as
$$
\partial_t w+Aw=F(w),
$$
where
$$
F(w)
:=
-P\operatorname{div}
\left[
w\otimes w+w\otimes\widehat u
+\widehat u\otimes w
\right].
$$
We then perform the cut-off procedure as follows. 
Let \(R>0\), and let
$
\chi_R\in C^\infty([0,\infty))
$
satisfy
\begin{equation}
\label{cut}
\begin{aligned}
0\leq\chi_R\leq1,& \qquad
\chi_R(r)=1\quad (0\leq r\leq R),\cr
\chi_R(r)=0\quad (r\geq2R),& 
\qquad
|\chi_R'(r)|\leq \frac{C_\chi}{R}.
\end{aligned}
\end{equation}
Define the cut-off nonlinearity
\begin{equation}\label{CutEq}
F_R(w):=
\chi_R(\|w\|_{H^1})F(w).
\end{equation}
We determine  the cut-off perturbation equation
\begin{equation}\label{2.16}
    \partial_t w+Aw=F_R(w).
\end{equation}
Since
\begin{equation}
\label{3.151}
F_R(w)=F(w)
\qquad\text{whenever }\|w\|_{H^1}\leq R,
\end{equation}
the cut-off and original perturbation equations coincide on the \(H^1\)-ball \(
B_{H^1}(0,R).\) 
 Moreover, we can see  \(F_R\) is globally Lipschitz from \(H^1\) into
\(H^{-1}\) with the Lipschitz constant $L_R$ determined latter in the proof of the main results.

\begin{theorem}[Inertial manifold around a periodic mild solution]\label{abstractResult}
\mbox{}\\
Let \(d>2\) and  $f\in C_b(\r, H^{1,1}(\Gamma(T\M\otimes T\M)))$  be  $T$-periodic.  Suppose that $\|f\|_{C_b}$ is 
  sufficiently small. Suppose that  \(\widehat u\) be the \(T\)-periodic mild solution of the Navier-Stokes equation \((2.1)\) obtained in Theorem 2.1. 
Assume that, for some \(N\in\mathbb N\), 
\begin{equation}\label{2.15}
{\sqrt{\lambda_{N+1}}-\sqrt{\lambda_N}}
>L_R.    
\end{equation}

\noindent
Then, the cut-off equation \eqref{2.16} 
possesses an \(N\)-dimensional inertial manifold.
\begin{remark}
    Note that, the above condition \eqref{2.15} is called spectral gap condition. Moreover,   since
 $$\sqrt{\lambda_{N+1}}-\sqrt{\lambda_N}
=
\frac{2N+d}
{\sqrt{(N+1)^2+N(d-1)}
+
\sqrt{N^2+(N-1)(d-1)}}$$ condition \eqref{2.15} is satisfied for certain choices of $N$ and $d$.
\end{remark}

\end{theorem}
\section{Proof of the Main Results}
To prove the main results, we begin by analyzing the corresponding linear equation, which serves as a foundational step in understanding the underlying dynamics.
\begin{equation}\label{oelin}
\begin{cases}
\dfrac{du}{dt}=\D u+ \p\dive f(t) \hbox{ for }t\in \r \\
u\in L^p_\sigma(\Gamma(T\M)),
\end{cases}
\end{equation} 
where  $\D$ is the Ebin-Marsden Laplacian as previously. As usual, using Duhamel's formula,   
by a \textit{mild solution} to \eqref{oelin} we mean a function $u$ satisfying the integral  equation:
\begin{equation}\label{InEq}
u(t) = e^{(t-s)\D}x + \int_s^te^{(t-\tau)\D}\p\dive f(\tau)d\tau \hbox{  for all  } t\geqslant s, \ t,s\in \r.
\end{equation}

We are now stating the following lemma on existence and uniqueness of  periodic mild solutions to linear equations.
\begin{lemma}[Periodic mild solutions to linear  equations]\label{mainlin}\mbox{}\\
 Suppose that  $f\in C_b(\r, L^{p/2}(\Gamma(T\M\otimes T\M)))$ is 
   $T$-periodic. Then,  there exists a unique initial vector field   $x^*\in L_\sigma^p(\Gamma(T\M)) $ such that
    the mild solution $\hat{u}$  to  \eqref{oelin} with initial vector field $x^*$ is 
  $T$-periodic. Moreover, the $T$-periodic mild solution $\hat{u}$  satisfies 
\begin{equation}\label{bdd}
  \|\hat{u}\|_{C_b(\r, L_{\sigma}^{p}(\Gamma(T\M))}\le \tilde{M}\|f\|_{C_b(\r, L^{p/2}(\Gamma(T\M\otimes T\M)))}\hbox{ for  }\tilde{M}:=\tilde C\beta^{ \frac{d-p}{2p}}\mathbf{\Gamma}(1-\theta)+\frac{1}{\beta}
\end{equation}
with $\mathbf{\Gamma}$ standing for the Gamma function.
\end{lemma}

\begin{proof}
The proof of this lemma may be done using Serrin's method (see \cite{Ser59} for the original approach, and  
see \cite{NV2025} for implementation  of Serrin's  method in framework of diffusion equations on spheres).   
Take 
 $u(t)=\int_s^te^{(t- \tau)\D}\p\dive f(\tau)d\tau $ for all $t\ge s$, 
 $t, s \in\r $. 
 Firstly, setting $X:=L_\sigma^p(\Gamma(T\M))$ and  using inequality \eqref{lplq1} 
we have
$$\|u(t)\|_{X}=\|\int^{t}_se^{(t-\tau)\D}\p\dive f(\tau)d\tau\|_{X}\le C_1\int^{t}_s\|e^{(t-\tau)\D}\dive f(\tau)\|_{X}d\tau
$$
$$\le
 C_1\int_s^{t}
 [h_d(t-\tau)]^{\frac{1}{p}+\frac{1}{d}}e^{-\beta(t-\tau)} \left\| f(\tau) \right\|_{L^{p/2}}d\tau \hbox{ for all }t\ge s.$$
 Since $h_d(t)= C \max \left( \frac{1}{t^{d/2}},1 \right)$,  it follows  from above 
inequality that
\begin{align*}
   \| u(t)\|_{X}&\le 
   C_1\int_s^{t} 
   \left([C\tau]^{-\frac d2 (\frac{1}{p}+\frac{1}{d})}+1\right) e^{-\beta(t-\tau)}d\tau\|f\|_{C_b(\r, L^{p/2}(\Gamma(T\M\otimes T\M)))}\\
   &\le  \left(\tilde C\beta^{ \frac{d-p}{2p}}\mathbf{\Gamma}(1-\theta)+\frac{1}{\beta}\right)\|f\|_{C_b(\r, L^{p/2}(\Gamma(T\M\otimes T\M)))}\hbox{ for all }t\ge s
   \end{align*}
   where  $\mathit{\mathbf{\Gamma}}$ is  the Gamma function,
   $\theta:=\frac d2 (\frac{1}{p}+\frac{1}{d})<1$ since  $p>d$, and
   $\tilde C:=C_1C^{-\frac d2 (\frac{1}{p}+\frac{1}{d})}$.
 Hence,  $u\in C_b([s, \infty), X)$ and
 \begin{equation}\label{esu}
 \| u\|_{C_b([s, \infty), X)}\le \tilde{M}\|f\|_{C_b(\r, L^{p/2}(\Gamma(T\M\otimes T\M)))}\hbox{ for }\tilde M:= \tilde C\beta^{ \frac{d-p}{2p}}\mathbf{\Gamma}(1-\theta)+\frac{1}{\beta}.
\end{equation}  
 
\def\n{\mathbb{N}} 
Next, we will prove that the sequence $\{{u}(s+nT)\}_{n\in \n}$ is a Cauchy sequence in $X$.   
Indeed,  set $w(t)={u}(t+(m-n)T)$ for any fixed couple of natural  numbers  $m>n \in \n$.  The periodicity of $f$ implies that 
 $w$ can be rewritten as
\begin{equation}\label{eqw}
w(t)=e^{(t- s)\D}{u}(s+(m-n)T)+\int_s^te^{-(t- \tau)\D}\p \dive f(\tau)d\tau\hbox{ for all }t\ge s.
\end{equation}

\def\T{\Delta_H}
 Now, it follows from \eqref{dispersive} that 
$$
\|{u}(t)-w(t)\|=\|e^{(t-s)\D}(u(s)-w(s))\|\le Me^{-\alpha (t-s)}\|u(s)-w(s)\|\le Ce^{-\alpha (t-s)}, \; \; t>s$$
for $C:=\frac{2M }{\alpha}\|f\|_{C_b(\r, L^{p/2}(\Gamma(T\M\otimes T\M)))}$ which is independent of $m$ and $ n$.

Substituting $t:=s+nT$ in the above inequality we have 
$
\|{u}(s+nT)-{u}(s+mT)\|\le Ce^{-\alpha nT}$ 
for all $m>n\in \n$. Therefore,
 $\{{u}(s+nT)\}_{n\in \n}$ is a Cauchy sequence in  $X$ which is a Banach space.
Hence,   
 $\{{u}(s+nT)\}_{n\in \n}$ is convergent in $X$, and we put
$x^*:= \lim_{n\to \infty}{u}(s+nT)\in X.$

\def\T{\D}
\medskip
Plugging  $x^*$ as the initial vector field, we
 now show that the mild solution 
$$\hat{u}(t)=
e^{(t-s)\D}{x^*}+\int_s^te^{(t-\tau)\T}g(\tau)d\tau$$
is $T$-periodic. 
To do this, we set
${v(t)}:=\int_s^{t+nT}e^{(t+nT-\tau)\T}\p\dive g(\tau)d\tau $ 
for arbitrary fixed $n\in \n$ and all $t\ge s$, that is, $v(t)=u(t+nT)$ for 
\begin{equation}\label{defu}
u(t)=\int_s^{t}e^{(t-\tau)\T}\p\dive f(\tau)d\tau.
\end{equation}
 Once again,  it follows from periodicity of $g$  that $v$ satisfies
$${v(t)}=e^{(t-s)\T}u(s
+nT)+\int_s^{t}e^{(t-\tau)\T}\p\dive f(\tau)d\tau $$
 for $u$ defined as in \eqref{defu}.

Thus, 
$
\|\hat{u}(s+T)-v(s+T)\|_X=\norm{e^{T\T}({\hat{u}(s)}-v(s))}_X\le
 \norm{e^{T\T}}\|{\hat{u}(s)}-v(s)\|_X.$

This implies 
 $
\|\hat{u}(s+T)-{u}(s+(n+1)T)\|_X\le \norm{e^{T\T}}\|{x^*}-{u}(s+nT)\|_X.
$
Letting $n\to \infty$ and  using the limit   
$\lim\limits_{n\to \infty}{u}(s+nT)=x^*=\hat{u}(s)$ in $X$ (see above) we have
 $\hat{u}(s+T)=\hat{u}(s).$  
We next prove  that  $\hat{u}(t+T)=\hat{u}(t)$ for all $t\ge s$. 
In fact, set  
$\hat{w}:=\hat{u}(t+T)$ for all $t\ge s$.  The periodicity of $f$  yields 
 $$\hat{w}(t)=e^{(t-s)\T}\hat{u}(s+T)+\int_s^te^{(t-s)\T}\p\dive f(\tau)d\tau\hbox{ for all }t\ge s.$$
 Since $\hat{u}(s+T)=\hat{u}(s)$ we have 
 $$\hat{w}(t)=e^{-(t-s)\T}\hat{u}(s)+\int_s^te^{-(t-s)\T}\p\dive f(\tau)d\tau=\hat{u}(t)\hbox{ for all }t\ge s.$$
 That is, $\hat{u}(t)=\hat{w}(t)=\hat{u}(t+T)$ for all $t\ge s$, implying that  
  $\hat{u}(t)$ is  $T$-periodic. The inequality $\eqref{bdd}$ can be obtained from  \eqref{esu} in combination with the properties  $\|u^*\|_X\le \|u\|_{C_b}$ and 
$\|\hat{u}\|_{C_b}=\sup_{s\le t\le s+T}\|\hat{u}(t)\|_X$ because of the periodicity of $\hat{u}$.

\noindent 
The uniqueness of the initial vector field $x^*$ and therefore of the $T$-periodic solution  follows from \eqref{dispersive}. 
Namely, suppose that $u$ and $v$ are two $T$-periodic mild  solutions of Equation  \eqref{oelin} subject to initial vector fields    $x_1^*=u_s$ and $x_2^*=v_s$, respectively. Then 
$u(t)-v(t)=e^{(t-s)\T}(x_1^*-x_2^*)$. From  \eqref{dispersive} we arrive at  $$\|u(t)-v(t)\|_X=\|e^{(t-s)\T}(x_1^*-x_2^*)\|_X\le Me^{-\alpha (t-s)}\|x_1^*-x_2^*\|_X.$$

Therefore, $\lim_{t\to \infty}\|u(t)-v(t)\|_X=0$. In combination with periodicity and continuity of $u$ and $v$, we have that $u(t)=v(t)$ for all $t\in \r$.
\end{proof}
\def\b{\mathcal{B}}
\def\TT{\mathcal{T}}
Then, we give the proof of our first main result on existence of a unique $T$-periodic mild solution to Eq. \eqref{NSESd}.
\begin{proof}[{\bf Proof of Theorem \ref{mainnon}}]
We consider the following ball
$$ 
\b^{per}_{\rho} =\{ v \in C_b(\r, L_\sigma^p(\Gamma(T\M))) :  v \, \text{ is $T$-periodic and } \, \|v \|_{C_b} \le \rho  \}. 
$$
Clearly, for each $v\in \b^{per}_{\rho}$, $v\otimes v$ belongs to $C_b(\r, L^{p/2}(\Gamma(T\M\otimes T\M)))$  and is  $T$-periodic.  
On $\b^{per}_{\rho}$,  we determine the transformation $\TT$ as follows:
Consider the equation 
\begin{equation}\label{ns1}
u(t)= e^{(t-s)\D}x+\int_s^t e^{-(t-\tau)\D}\p \, \dive(-(v\otimes v)(\tau)+ f(\tau))d\tau\hbox{\ for all }t\ge s.
\end{equation}
Then, for $v\in  \B_\rho^{T}$ we put 
\begin{equation}\label{defphi}
\begin{split}
\TT(v)&=u\hbox{ where 
$u\in C_b(\r, X)$ is the unique $T$-periodic solution}\cr
& \hskip 5cm\hbox{  to Equation \eqref{ns1},}
\end{split}
\end{equation}
where $X:=L_\sigma^p(\Gamma(T\M))$, as previously. 
The existence and uniqueness of the above  periodic solution $u$ is proved in Lemma  \ref{mainlin}. Hence, $\TT$ is well-defined.

Furthermore,  
using the fact that $v\in \b^{per}_\rho$, by \eqref{bdd} and H\"older inequality we derive
$$
\| u \|_{C_b(\r, X}) \le  \tilde{M}(\| v\otimes v \|_{C_b(\r, L^{p/2}(\Gamma(T\M\otimes T\M)))} + \|f\|_{C_b(\r, L^{p/2}(\Gamma(T\M\otimes T\M)))})\le$$
$$\le \tilde{M}(\rho^2+\|f\|_{C_b(\r, L^{p/2}(\Gamma(T\M\otimes T\M)))}).
$$
Hence, $\TT$ maps from $\b^{per}_{\rho}$ into itself provided $\| f \|_{C_b(\r, L^{p/2}(\Gamma(T\M\otimes T\M)))}$ and  $\rho$ are sufficiently small.

For $v_1, v_2 \in \b^{per}_{\rho}$ set  $\TT v_1:=z_1$,  $\TT v_2:=z_2$.  From H\"older's inequality it follows that   
 \begin{align}
\| z_1 - z_2 \|_{C_b(\r, X)} &\le \tilde{M}\| v_1\otimes v_1 - v_2\otimes v_2 \|_{C_b(\r, L^{p/2}(\Gamma(T\M\otimes T\M)))}  \notag \\
&=\tilde{M}\| v_1\otimes v_1 - v_1\otimes v_2+v_1\otimes v_2- v_2\otimes v_2 \|_{C_b(\r, L^{p/2}(\Gamma(T\M\otimes T\M)))} \cr
&\le  2 \tilde{M}  \rho  \| v_1 - v_2 \|_{C_b(\r, X)}. \notag
\end{align} 
Thus, if  $\rho$ and $\| f \|_{C_b(\r, L^{p/2}(\Gamma(T\M\otimes T\M)))}$  
are sufficiently small, then $\TT$ is a selfmap in $\b^{per}_{\rho}$, and  a contraction. Therefore, $\TT$ has a unique fixed point, and this is the unique  periodic mild solution to equation \eqref{NSESd}
in a small closed ball of $C_b(\r, L_\sigma^{p}(\Gamma(T\M)))$.\end{proof}

We now give the  proof of our main result stated in Theorem \ref{abstractResult}.
\begin{proof}[{\bf  Proof of Theorem \ref{abstractResult} (Inertial Manifolds)}]
We remark first that, since $\M$ is compact, the Navier-Stokes equations on $S^d$ is dissipative as proved below:
Indeed, for the Navier-Stokes equation
\begin{equation}
    u_t+\nabla_u u+\operatorname{grad}\pi
    =\Delta u+\operatorname{div}f,
    \qquad
    \operatorname{div}u=0,
\end{equation}
taking the \(L^2(\M)\)-inner product with \(u\) gives
\begin{equation}\label{3.171}
    \frac{1}{2}\frac{d}{dt}\|u(t)\|_{L^2}^2
    +\|\nabla u(t)\|_{L^2}^2
    =-(f(t),\nabla u(t)).
\end{equation}

Indeed,
\begin{equation}
\begin{aligned}
    (\nabla_u u,u)
    &=\frac{1}{2}\int_{\M}u\cdot\nabla |u|^2\,dS =-\frac{1}{2}\int_{\M}
(\operatorname{div}u)|u|^2\,dS
    =0.
\end{aligned}
\end{equation}

Since \(S^d\) is compact and the  Ebin-Marsden
Laplacian has strictly positive first eigenvalue,
\[
    \|\nabla u\|_{L^2}^2
    \geq \lambda_1\|u\|_{L^2}^2\hbox{ for }\lambda_1>0.
\]
Hence
\begin{equation}    \label{3.18}
    \frac{d}{dt}\|u(t)\|_{L^2}^2
+\lambda_1\|u(t)\|_{L^2}^2
    \leq C\|f(t)\|_{L^2}^2.
\end{equation}
Thus the equation is dissipative in \(L^2\), that is, bounded sets are
eventually absorbed into a bounded \(L^2\)-ball.

For \(f=0\), in particular,
\begin{equation}
    \|u(t)\|_{L^2}
    \leq e^{-\lambda_1t/2}\|u(0)\|_{L^2}.
    \tag{3.19}
\end{equation}
This is valid for every \(d\geq2\).

Then, since $f\in C_b(\r, H^{1,1}(\Gamma(T\M\otimes T\M)))$, in a similar way as in Theorem \ref{mainnon}, we have that the periodic solution $\hat{u}$ belongs to $C_b(\r, H^{1}(\Gamma(T\M)))$
We next  consider equation for the perturbation around the periodic solution. As previously, take any solution
 $u$ of Equation \eqref{NSESd}  and put 
$
w(t):=u(t)-\widehat u(t).
$
Since both \(u\) and \(\widehat u\) satisfy the Navier-Stokes equation \eqref{NSESd},  subtraction gives
$$
\partial_t w
+
P\operatorname{div}
\left[
(w+\widehat u)\otimes(w+\widehat u)
-\widehat u\otimes\widehat u
\right]
=
\D w.
$$
Hence, with
$
A:=-\D,
$ 
we obtain 
\begin{equation}\label{3.9}
\partial_t w+Aw=F(w)
\end{equation}
where
\begin{equation}\label{3.10}
F(w)
=
-P\operatorname{div}
\left[
w\otimes w+w\otimes\widehat u
+\widehat u\otimes w
\right].    
\end{equation}
Clearly,
 $
F(0)=0.
 $ 
We next prove the local Lipschitz continuity of $F$. To this purpose,
 for \(v,z\in H^1(\Gamma(T\M))\), Sobolev embedding yields
 $
H^1(\Gamma(TS^d))
\hookrightarrow
L^{\frac{2d}{d-2}}(\Gamma(TS^d)).
$ 
Therefore,
\begin{equation}\label{3.11
}
    \begin{aligned}
\|P\operatorname{div}(v\otimes z)\|_{H^{-1}}
&\leq C\|v\otimes z\|_{L^2}\leq
C\|v\|_{L^{\frac{2d}{d-2}}}
\|z\|_{L^{\frac{2d}{d-2}}}\leq C\|v\|_{H^1}\|z\|_{H^1}.
\end{aligned}
\end{equation}

Consequently, for \(w_1,w_2\in H^1\),

\begin{equation}\label{3.12}
\begin{aligned}
\|F(w_1)-F(w_2)\|_{H^{-1}}
\leq
C\bigl(
\|w_1\|_{H^1}
+\|w_2\|_{H^1}
+\|\widehat u\|_{H^1}
\bigr)
\|w_1-w_2\|_{H^1}.
\end{aligned}
\end{equation}

Thus \(F\) is locally Lipschitz from \(H^1\) into
\(H^{-1}\).


We then consider the cut-off equations \eqref{CutEq} and  
prove that the cut-off nonlinear operator $F_R$ defined as in \eqref{CutEq} is  globally  Lipschitz. To do this,  
fix \(R>0\). Let
$\chi_R$ by the cut-off function as in \eqref{cut}. Recall that  
$
F_R(w)
=
\chi_R(\|w\|_{H^1})F(w).
$ 
By construction,
$
F_R(w)=F(w)
\text{ if }\|w\|_{H^1}\leq R.
$

Let \(w_1,w_2\in H^1\). We write
\begin{equation}\label{3.16}
    \begin{aligned}
F_R(w_1)-F_R(w_2)
={}&
\chi_R(\|w_1\|_{H^1})
\bigl(F(w_1)-F(w_2)\bigr)
+
\bigl[
\chi_R(\|w_1\|_{H^1})
-\chi_R(\|w_2\|_{H^1})
\bigr]F(w_2).
\end{aligned}
\end{equation}
If both \(\|w_1\|_{H^1}\) and \(\|w_2\|_{H^1}\) are at most \(2R\), then (3.12) gives
\begin{equation}\label{3.17}
    \|F(w_1)-F(w_2)\|_{H^{-1}}
\leq
C(4R+K)\|w_1-w_2\|_{H^1},
\hbox{ where }
K:=\|\widehat u\|_{C_b(\mathbb R,H^1)}.
\end{equation}
Moreover,
$$
\begin{aligned}
&
\left|
\chi_R(\|w_1\|_{H^1})
-\chi_R(\|w_2\|_{H^1})
\right|
\leq
\frac{C_\chi}{R}
\|w_1-w_2\|_{H^1}.
\end{aligned}
$$

Since the second term in (\ref{3.16}) is non-zero only when
\(\|w_2\|_{H^1}\leq2R\), we also have
\begin{equation}\label{3.19}
\|F(w_2)\|_{H^{-1}}
\leq C(2R+K)^2.
\end{equation}
Combining (\ref{3.16})--(\ref{3.19}), we obtain
\begin{equation}
 \label{3.20}   
{
\|F_R(w_1)-F_R(w_2)\|_{H^{-1}}
\leq
L_R\|w_1-w_2\|_{H^1},
}
\end{equation}
where
$
L_R
=
C(4R+K)
+
\frac{C_\chi}{R}C(2R+K)^2.
$ 
Thus \(F_R\) is globally Lipschitz from \(H^1\) into
\(H^{-1}\).

We now can apply the  Zelik's theorem \cite[Theorem 2.15]{Ze14} to
 $$
\partial_t w+Aw=F_R(w).
$$
 Here the phase space is
$
H=L^2_\sigma(\Gamma(TS^d)),
$ 
and
 $
H^1=D(A^{1/2}),
\qquad
H^{-1}=D(A^{-1/2}).
$  
Thus the nonlinearity \(F_R\) is globally Lipschitz from \(H^1\) into
\(H^{-1}\), corresponding to the case \(\beta=-1\) in Zelik's theorem.

The spectral-gap condition (2.65) of \cite[Theorem 2.15]{Ze14} therefore becomes
$$
\sqrt{\lambda_{N+1}}
-
\sqrt{\lambda_N}
>L_R.
$$
For the Ebin--Marsden Laplacian,
$$
\lambda_k=k^2+(k-1)(d-1),
\qquad k=1,2,\ldots.
$$
Hence
\begin{equation} \label{3.26}
\begin{aligned}
\sqrt{\lambda_{N+1}}-\sqrt{\lambda_N}
=
\frac{2N+d}
{\sqrt{(N+1)^2+N(d-1)}
+
\sqrt{N^2+(N-1)(d-1)}}.
\end{aligned}
\end{equation}
Whenever \(N\) and \(d\) are chosen so that
$$
\sqrt{\lambda_{N+1}}-\sqrt{\lambda_N}>L_R,
$$

Zelik's Theorem \cite[Theorem 2.15]{Ze14} yields an \(N\)-dimensional inertial manifold
$$
\mathcal M_R
=
\{w_++\Phi_R(w_+):w_+\in P_NH\}
$$
for the cut-off equation (3.22), where
$$
\Phi_R:P_NH\to Q_NH
$$
is Lipschitz continuous and the exponential tracking property holds.

Lastly, we would like to explain 
the relation with the original Navier-Stokes equation. To this end, 
by \eqref{3.151}, the cut-off equation and the original perturbation equation coincide on
$$
B_R:=\{w\in H^1:\|w\|_{H^1}\leq R\}.
$$
Therefore, every trajectory of the cut-off equation which remains in \(B_R\) also solves the original perturbation equation.

In particular, if $
\mathcal U\subset B_R
$ 
is a positively invariant neighborhood of \(w=0\), then the part 
$
\mathcal M_R\cap\mathcal U
$ 
is invariant for the original perturbation equation and provides a finite-dimensional reduction of the Navier-Stokes dynamics in \(\mathcal U\).
 This completes the proof.
\end{proof}



\end{document}